\documentclass[a4paper,11pt]{article}

\usepackage{amsmath,amscd,amssymb,amsthm,amssymb}
\usepackage{enumitem}

\usepackage{hyperref}
\usepackage{cite}

\usepackage[super]{nth}

\usepackage[british, USenglish]{babel}
\usepackage{tikz-cd}

\usepackage{fullpage}

\newtheorem{thm}{Theorem}[section]
\newtheorem{prop}[thm]{Proposition}
\newtheorem{lem}[thm]{Lemma}

\theoremstyle{definition}
\newtheorem{rem}[thm]{Remark}

\newcommand{\e}{\varepsilon}

\newcommand{\NN}{\mathbb N}

\newcommand{\ZZ}{\mathbb Z}
\newcommand{\RR}{\mathbb R}
\newcommand{\CC}{\mathbb C}

\newcommand{\PP}{\mathbb P}
\newcommand{\QQ}{\mathbb Q}

\newcommand{\lr}{\longrightarrow}

\title{Almost complex structures on connected sums of complex projective spaces}
\author{Oliver Goertsches and Panagiotis Konstantis\footnote{Fachbereich Mathematik und Informatik, Philipps--Universit\"at Marburg, Hans-Meerwein-Stra\ss e 6, 35032 Marburg. Email: \texttt{\href{mailto:goertsch@mathematik.uni-marburg.de}{goertsch@mathematik.uni-marburg.de}}, \texttt{\href{mailto:pako@mathematik.uni-marburg.de}{pako@mathematik.uni-marburg.de}}}}
\date{}

\begin{document}

\maketitle

\begin{abstract}
  We show that the $m$-fold connected sum $m\# \CC\PP^{2n}$ admits an almost complex structure if and only if $m$ is odd.
\end{abstract}
\section{Introduction}
A \emph{complex structure} on a real vector bundle $F$ over a connected CW complex $X$ is
a complex vector bundle $E$ over $X$ such that its underlying real vector bundle $E_\RR$ is
isomorphic to $F$. A \emph{stable complex structure} on $F$ is a complex
structure on $F \oplus \e^d$, where $\e^d$ is the $d$ dimensional trivial real vector bundle over $X$.
For $X$ a manifold we say that $X$ has an \emph{almost complex structure} (respectively
\emph{stable almost complex structure}) if its tangent bundle admits an complex structure
(respectively stable complex structure).
Motivated by the question in \cite{mathoverflow} we consider in this paper the $m$-fold connected sum of complex projective spaces $m\#\CC\PP^{2n}$.

As shown by Hirzebruch \cite[Kommentare, p.\ 777]{Hirzebruch}, a necessary condition for the existence of an almost complex structure on a $4n$-dimensional compact manifold $M$ is the congruence $\chi(M) \equiv (-1)^n\sigma(M) \text{ mod } 4$, where $\chi(M)$ is the Euler characteristic and $\sigma(M)$ the signature of $M$. Thus, for even $m$, the connected sums above cannot carry an almost complex structure. We will show that for odd $m$ they do admit almost complex structures, thus showing

\begin{thm}\label{T:Maintheorem} The $m$-fold connected sum $m\# \CC\PP^{2n}$ admits an almost complex structure if and only if $m$ is odd.
\end{thm}
In odd complex dimensions, the connected sums $m\# \CC\PP^{2n+1}$ are K\"ahler, since $\CC\PP^{2n+1}$ admits an
orientation reversing diffeomorphism and therefore $m \# \CC\PP^{2n+1}$ is diffeomorphic to $\CC\PP^{2n+1}\# (m-1)\overline{\CC\PP^{2n+1}}$ which is a blow--up of $\CC\PP^{2n+1}$ in $m-1$ points, hence K\"ahler. Furthermore Theorem \ref{T:Maintheorem} is known for $n=1$   and $n=2$, see \cite{Audin1991} and
\cite{GeigesMueller2000} respectively. In both cases the authors use general results on the
existence of almost complex structures on manifolds of dimension $4$ and $8$ respectively.

We will prove Theorem \ref{T:Maintheorem} as follows. In \cite[Theorem 1.1]{Sutherland1965} or in
\cite[Theorem 1.7]{Thomas1967} the authors showed
\begin{thm}\label{T:SACSandACS}
Let $M$ be a closed smooth $2d$-dimensional manifold. Then $TM$ admits an almost complex structure if and only if it admits a stable almost complex structure $E$ such that $c_d(E) = e(M)$, where $c_d$ is the $d$--th Chern class of $E$ and $e(M)$ is the Euler class of $M$.
\end{thm}
In Section \ref{S:Stable_almost_complex_structures}
we will describe the full set of stable almost complex structures in the reduced
$K$--theory of $m\#\CC\PP^{2n}$. In Section \ref{S:modd} we give, for odd $m$, an explicit example of a stable almost complex structure to which Theorem \ref{T:SACSandACS} applies. \\

\noindent {\it Acknowledgements} We wish to thank Thomas Friedrich for valuable comments on an earlier version of the paper. We are also grateful to the anonymous referee for his careful reading and helpful comments.

\section{Stable almost complex structures on $m\#
\CC\PP^{2n}$}\label{S:Stable_almost_complex_structures}

For a CW complex $X$ let $K(X)$ and ${KO}(X)$ denote the complex and real $K$--groups respectively.
Moreover we denote by $\widetilde K^{}\left(X\right)$ and $\widetilde{KO}^{}\left(X\right)$ the
reduced groups. Let $r \colon K(X) \to KO(X)$ denote the real reduction map, which can be
restricted to a map $\widetilde K^{}\left(X\right)\to \widetilde{KO}^{}\left(X\right)$. We denote
the restricted map again with $r$. A real vector bundle $F$ over $X$ has a stable almost complex
structure if there is a an element $y \in \widetilde K(X)$ such that $r(y)=F-\dim F$. Since $r$
is a  group homomorphism, the set of all stable complex vector bundles, such that the underlying
real vector bundle is stably isomorphic to $F$, is given by
\[
 y + \ker r \subset \tilde K(X),
\]
where $y$ is such that $r(y) =F -\dim F$. Let $c \colon KO(X) \to K(X)$ denote the complexification
map and $t \colon K(X) \to K(X)$ the map which is induced by complex conjugation of complex
vector bundles. The maps $t$ and $c$ are ring homomorphisms, but $r$ preserves only the group
structure. The following idendities involving the maps $r,c$ and $t$ are well known
\begin{align*}
  c\circ r &= 1+ t \colon K(X) \to K(X),\\
  r\circ c &= 2 \colon KO(X) \to KO(X).
\end{align*}
We will write $\bar y=t(y)$ for an element $y \in K(X)$.

For two oriented manifolds $M$ and $N$ of same dimension $d$, we denote by $M\# N$ the connected
sum of $M$ with $N$ which inherits an orientation from $M$ and $N$. First, let us
characterise the stable tangent bundle of
$M\# N$ by

\begin{lem}\label{L:StableTangentBundleConnectedSum}
  Let $p_M\colon M\# N \to M$ and $p_N \colon M\#N \to N$ be collapsing maps to each factor of $M\#
  N$. Then we have
  \[
    p_M^\ast(M)\oplus p_N^\ast(N) \cong T(M\#N)\oplus \varepsilon^d.
  \]
\end{lem}

\begin{proof}
  Let $D_M \subset M$ and $D_N \subset N$ be embedded closed disks and $W_M$ and $W_N$ collar
  neighborhoods of $\partial(M \setminus  \mathring D_M)$ and $\partial(N \setminus \mathring D_N)$
  respectively, where $\mathring D$ denotes the interior of $D$. Thus $W_M \cong S^{d-1}\times [-2,0]$
  and $W_N \cong S^{d-1} \times [0,2]$. The manifold $M\#N$ is obtained by identifying
  $S^{d-1}\times 0 \subset W_M$ with $S^{d-1}\times 0\subset W_N$ by the identity map.
  Set $W := W_M \cup W_N \subset M\# N$ and note that $V_1:=p_M^\ast(M)\oplus p_N^\ast(N)$ as well
  as $V_2:=T(M\# N)\oplus \varepsilon^n$ are trivial over $W$. Moreover let $U_M\subset M\#N$ be the open set
  diffeomorphic to $(M\setminus W_M)\cup (S^{d-1}\times [-2,-1[)$ and analogous for $U_N
  \subset M\# N$.

  Now, since $V_1|_{U_M} \cong p_M^\ast(TM)\oplus\varepsilon^d$ and $p^\ast_M(TM)|_{U_M} = T(M\#
  N)|_{U_M}$ we have an isomorphism given by $\Phi_M \colon V_2|_{U_M} \to V_1|_{U_M}$, $(\xi,w)\mapsto
  ((p_M)_\ast(\xi),w)$. For $\Phi_N \colon V_2|_{U_N} \to V_1|_{U_N}$, we set $\Phi_N(\eta,w)
  =(w,-(p_N)_\ast(\eta))$. Moreover both vector bundles $V_1$ and $V_2$ are trivial over $W$ and
  it is possible to choose trivializations of $V_1$ and $V_2$ over $W$ such that $\Phi_M$ is
  given by $(v,w)\mapsto (v,w)$ over $W_M$ and such that $\Phi_N$ is represented by $(v,w)\mapsto
  (w,-v)$ over $W_N$. Over $S^{d-1}\times [-1,1]$ we can interpolate these isomorphisms by
  \[
    \begin{pmatrix}
      v \\
      w
    \end{pmatrix}\mapsto
    \begin{pmatrix}
      \cos\left(\frac{\pi}{4}(t+1)\right) & \sin\left(\frac{\pi}{4}(t+1)\right)\\
      -\sin\left(\frac{\pi}{4}(t+1)\right) & \cos\left(\frac{\pi}{4}(t+1)\right)\\
    \end{pmatrix}
    \begin{pmatrix}
      v \\
      w
    \end{pmatrix}
  \]
  for $t\in [-1,1]$. Using this interpolation we can glue $\Phi_M$ and $\Phi_N$ to a global isomorphism $V_2\to V_1$.
\end{proof}

Hence $T(M\#N) -d = TM + TN -2d$ in $\widetilde{KO}(M\# N)$, where $TM$ and $TN$ denote the
elements in $\widetilde{KO}(M\# N)$ induced by $p_M^*(TM)$ and $p_N^*(TN)$ respectively. This shows
that if $M$ and $N$ admit stable almost complex structures so does $M \# N$ (cf. \cite{MR0258037}).
For $M=N=\CC\PP^{2n}$ we consider the natural orientation induced by the complex structure of
$\CC\PP^{2n}$.

We proceed with recalling some basic facts on complex projective spaces. Let $H$ be the tautological
line bundle over $\CC\PP^{d}$ and let $x \in H^2(\CC\PP^{d};\ZZ)$ be
the generator, such that the total Chern class $c(H)$ is given by $1+x$. The cohomology ring of
$\CC\PP^{d}$ is isomorphic to $ \ZZ[x]/\langle x^{d+1}  \rangle$.
The $K$ and $KO$ theory of $\CC\PP^{d}$ are completely understood. Let $\eta:= H-1
\in \widetilde K^{}\left(\CC\PP^{d}\right)$ and
$\eta_R:=r(\eta) \in \widetilde{KO}(\CC\PP^{d})$. Then we have \begin{thm}[cf.\
  {\cite[Theorem 3.9]{Sanderson1964}}, {\cite[Lemma 3.5]{MR0202131}}, {\cite[p.\ 170]{MR0440554}} and {\cite[Proposition 4.3]{Thomas1974}}]\label{T:KtheoryOfComplexProjectiveSpace}~
    \begin{enumerate}[label=(\alph*)]
      \item $K(\CC\PP^{d}) = \ZZ[\eta]/\langle \eta^{d+1}\rangle$. The following sets of elements are an
              integral basis of $K(\CC\PP^{d})$
                \begin{enumerate}[label=(\roman*)]
                  \item $1,\,\eta,\, \eta(\eta+\bar\eta),\, \ldots,\, \eta(\eta +\bar\eta)^{n-1},
                      (\eta+\bar\eta),\, \ldots,\, (\eta + \bar\eta)^{n}$, and also, in case
                      $d$ is odd, $\eta^{2n+1} = \eta(\eta + \bar\eta)^n$.
                  \item
                  $1,\,\eta,\, \eta(\eta+\bar\eta),\, \ldots,\, \eta(\eta +\bar\eta)^{n-1},
                    (\eta-\bar\eta)(\eta+\bar\eta),\, \ldots,\, (\eta-\bar\eta)(\eta + \bar\eta)^{n-1}$,
                    and also, in case $d$ is odd, $\eta^{2n+1}$
                \end{enumerate}
                where $n$ is the largest integer $\leq d/2$.
           \item
              \begin{enumerate}[label=(\roman*)]
                \item if $d=2n$ then $KO(\CC\PP^{d}) = \ZZ[\eta_R]/\langle \eta_R^{n+1}\rangle$
                \item if $d=4n+1$ then $KO(\CC\PP^{d}) = \ZZ[\eta_R]/\langle \eta_R^{2n+1},2\eta_R^{2n+2}
                \rangle$
              \item if $d=4n+3$ then $KO(\CC\PP^{d}) = \ZZ[\eta_R]/\langle \eta^{2n+2}_R  \rangle$.
              \end{enumerate}

      \item The complex stable tangent bundle is given by $(2n+1)\bar\eta \in \tilde
        K(\CC\PP^{2n}$) and the real stable tangent bundle is given by $r( (2n+1)\bar\eta)) \in \widetilde{KO}^{}\left(\CC\PP^{2n}\right)$.

      \item The kernel of the real reduction map $r \colon \widetilde K^{}\left(\CC\PP^{d}\right) \to
        \widetilde{KO}^{}\left(\CC\PP^{d}\right)$ is freely generated by the elements
        \begin{enumerate}[label=(\roman*)]
          \item $\eta-\bar\eta, (\eta-\bar\eta)(\eta+\bar\eta),\ldots,
            (\eta-\bar\eta)(\eta+\bar\eta)^{\tfrac{d}{2}-1}$, if
            $d$ is even,
          \item $\eta-\bar\eta, (\eta-\bar\eta)(\eta+\bar\eta),\ldots,(\eta-\bar\eta)
            (\eta+\bar\eta)^{2n-1}, 2\eta^d$,
            if $d = 4n+1$,
          \item $\eta-\bar\eta, (\eta-\bar\eta)(\eta+\bar\eta),\ldots,(\eta-\bar\eta)
            (\eta+\bar\eta)^{2n}, \eta^d$,
            if $d = 4n+3$.
        \end{enumerate}
    \end{enumerate}
\end{thm}
Next we would like to describe the integer cohomology ring of $m\#\CC\PP^{2n}$. For that we
introduce the following notation: Let $\Lambda$ denote either $\ZZ$ or $\QQ$. We define an ideal
$R_d(X_1,\ldots,X_m)$ in $\Lambda[X_1,\ldots,X_m]$, where $X_1,\ldots,X_m$ are indeterminants, as
the ideal generated by the following elements
  \begin{align*}
    X_i \cdot X_j &,\quad  i\neq j\\
    X_i^{d} -X_j^{d} &,\quad i\neq j,\\
    X_j^{d+1} &,\quad j=1,\ldots,m.
  \end{align*}
Hence we have
  \begin{equation}
    H^*\left( m\#\CC\PP^{d};\Lambda  \right)\cong \Lambda[x_1,\ldots,x_m]/R_{d}(x_1,\ldots,x_m)
    \label{Eq:CohomologyOfConnectedSum}
  \end{equation}
where $x_j = p^*_j(x) \in H^2\left( m\#\CC\PP^{d};\Lambda  \right)$, for $x \in
H^2(\CC\PP^{d};\Lambda)$ defined as above and $p_j\colon m\#\CC\PP^{d} \to \CC\PP^{d}$ the
projection onto the $j$-th factor. Note that $p_j$ induces an monomorphism on cohomology.

The stable tangent bundle of $m\#\CC\PP^{2n}$ in
$\widetilde{KO}(m\#\CC\PP^{2n})$ is represented by
\[
  (2n+1)\sum_{j=1}^m r(\bar\eta_j)
\]
where $\eta_j := p_j^*(\eta) \in \widetilde K^{}\left(\CC\PP^{2n}\right) $ and
$r\colon \widetilde K(m\#\CC\PP^{2n}) \to \widetilde{KO}(m\#\CC\PP^{2n})$ is the real reduction map.
Hence the set of stable almost complex structures on $m\#\CC\PP^{2n}$ is given by
\begin{equation}\label{Eq:Kernel_of_r}
  (2n+1)\sum_{j=1}^m \bar\eta_j + \ker r,
\end{equation}

\noindent
For $k \in \NN$ and $j=1,\ldots,m$, set $w_j^k=p_j^*(H)^k - p_j^*(H)^{-k}$,
$e_j^{n-1}=\eta_j(\eta_j + \bar\eta_j)^{n-1}$ and $\omega = \eta_1^{2n}$.
\begin{prop}
  The kernel of $r\colon \widetilde K^{}\left(m\#\CC\PP^{2n}\right)\to \widetilde{KO}^{}\left(m\# \CC\PP^{2n}\right)$ is freely generated by
  \begin{enumerate}[label=(\alph*)]
    \item $\{w_j^k : k=1,\ldots,n-1,\, j=1,\ldots,m \} \cup \{e^{n-1}_1-e^{n-1}_{j} : j=
      2,\ldots,m\} \cup \{2e_1^{n-1} -\omega\}$, for $n$ even,
    \item  $\{ w_j^k : k=1,\ldots,n,\, j=1,\ldots,m \}$, for $n$ odd.
  \end{enumerate}
\end{prop}

\begin{proof}
  Consider  the cofiber sequence
    \begin{equation}
      \bigvee_{j=1}^m \CC\PP^{2n-1} \stackrel{i}{\lr} m\# \CC\PP^{2n} \stackrel{\pi}{\lr} S^{4n}.
      \label{Eq:Cofibration}
    \end{equation}
  Note that the line bundle $i^*p_j^*(H)$ is the tautological line bundle over the $j$-th summand
  of $\vee_{j=1}^m \CC\PP^{2n-1}$ and the trivial bundle on the other summands, since the first
  Chern classes are the same. For the reduced groups we have
  \[
    \widetilde K^{}\left(\vee_{j=1}^m \CC\PP^{2n-1}\right) \cong \bigoplus_{j=1}^m
    \widetilde K^{}\left(\CC\PP^{2n-1}\right)
  \]
  and $i^*p_j^*(\eta)$ generates the $j$-th summand of the above sum according to Theorem
  \ref{T:KtheoryOfComplexProjectiveSpace}.
  The long exact sequence in $K$-theory of the cofibration \eqref{Eq:Cofibration} is given by
    \begin{equation}\label{Eq:LES}
      \cdots \to \widetilde K^{-1}\left(\vee_{j=1}^m \CC\PP^{2n-1}\right)\to
      \widetilde K^{}\left(S^{4n}\right) \to \widetilde K^{}\left(m\#\CC\PP^{2n}\right)\to
      \widetilde K^{}\left(\vee_{j=1}^m \CC\PP^{2n-1}\right)\to  \widetilde
      K^{1}\left(S^{4n}\right)\to \cdots
    \end{equation}
  From Theorem $2$ in \cite{Fujii1967} we have $\widetilde K^{-1}\left(\CC\PP^{2n-1}\right)=0$,
  hence $\widetilde K^{-1}\left(\vee_{j=1}^m \CC\PP^{2n-1}\right)=0$ and from Bott periodicity we
  deduce $\widetilde K^{1}\left(S^{4n}\right)=\widetilde K^{-1}\left(S^{4n}\right)=0$. So we obtain a short exact sequence
  \[
    0 \lr \widetilde K^{}\left(S^{4n}\right) \stackrel{\pi^*}{\lr}
    \widetilde K^{}\left(m\#\CC\PP^{2n}\right)\stackrel{i^*}{\lr}
    \widetilde K^{}\left(\vee_{j=1}^m \CC\PP^{2n-1}\right)\lr  0.
  \]
  which splits, since the involving groups are finitely generated, torsion
  free abelian groups.
  Let $\omega_\CC$ be the generator of $\widetilde K^{}\left(S^{4n}\right)$, then the set
  \[
    \left\{\pi^*(\omega_\CC)\right\}\cup
    \left\{\eta_j^k: j=1,\ldots,m,\; k=1,\ldots,2n-1\right\}
  \]
  is an integral basis of $\widetilde K^{}\left(m\# \CC\PP^{2n}\right)$. We claim that
  $\eta_j^{2n}=\pi^*(\omega_\CC)$ for all $j$. Indeed, the elements $\eta_j^{2n}$ lie in the kernel
  of $i^*$, hence there are $k_j \in \ZZ$ such that $\eta_j^{2n}=k_j \cdot \pi^*(\omega_\CC)$. Let
  $\widetilde{ch} \colon \widetilde K^{}\left(X\right) \to \widetilde{H}^{}\left(X;\QQ \right)$
  denote the Chern character for a finite CW complex $X$, then $\widetilde{ch}$ is a monomorphism for
  $X=m\#\CC\PP^{d}$ (since $\tilde H^{*}(m\#\CC\PP^{d};\ZZ)$ has no torsion, cf. \cite{MR0121801})
  and an isomorphism for $X=S^{d}$ onto $\widetilde H^{*}(S^{d};\ZZ)$ embedded in $\widetilde
  H^{*}(S^{d};\QQ)$. Using the notation of \eqref{Eq:CohomologyOfConnectedSum} we have
  \[
    \widetilde{ch}(\eta_j^{2n}) = \left( e^{x_j}-1  \right)^{2n} = x_j^{2n}
  \]
  and using the naturality of $\widetilde{ch}$
  \[
    \widetilde{ch}\left( \pi^*(\omega_\CC)  \right) = \pi^*\left( \widetilde{ch}(\omega_\CC) \right)
    =\pm x_j^{2n}
  \]
  since $\pi^*$ is an isomorphism on cohomology in dimension $2n$. We can choose $\omega_\CC$ such
  that $\widetilde{ch}(\pi^*(\omega_\CC))=x_j^{2n}$. This shows $k_j=1$ for all $j$ and
  $\widetilde K^{}\left(m\#\CC\PP^{2n}\right)$ is freely generated by
  \[
    \left\{\eta_j^k: j=1,\ldots,m,\; k=1,\ldots,2n-1\right\}
    \cup \{\eta_1^{2n}= \cdots = \eta_m^{2n}\}.
  \]
  Hence $K(m\#\CC\PP^{2n})=\ZZ[\eta_1,\ldots,\eta_m]/R_{2n}(\eta_1,\ldots,\eta_m)$. Since $p_j^*(H)\otimes p_j^*(\overline H)$ is the trivial bundle we compute
  the identity
  \[
    \overline{\eta}_j = \frac{-\eta_j}{1+\eta_j}=-\eta_j+\eta_j^2
    - \cdots +\eta_j^{2n}.
  \] The ring $\ZZ[\eta_1,\ldots,\eta_m]/R_{2n}(\eta_1,\ldots,\eta_m)$ is isomorphic to
  \[
    \left. \left( \bigoplus_{j=1}^m \ZZ[\eta_j]/\langle \eta_j^{2n+1}  \rangle   \right)
    \right/\langle \eta_j^{2n} -\eta_i^{2n} : j\neq i  \rangle
  \]
  and from Theorem \ref{T:KtheoryOfComplexProjectiveSpace} the set $\Gamma_j$ which contains the
  elements
  \begin{align*}
    &\eta_j,\, \eta_j(\eta_j+\overline \eta_j),\ldots,\eta_j(\eta_j+\overline\eta_j)^{n-1}\\
    &\eta_j-\overline\eta_j,
    (\eta_j-\overline\eta_j)(\eta_j+\overline\eta_j),\ldots,(\eta_j-\overline\eta_j)
    (\eta_j+\overline\eta_j)^{n-1}
  \end{align*}
  together with $\{1\}$ is an integral basis of $\ZZ[\eta_j]/\langle \eta_j^{2n+1} \rangle$. Thus
  the set $\Gamma_1 \cup \ldots \cup \Gamma_m \subset \widetilde K(m\#\CC\PP^{2n})$ generates
  the group $\widetilde K(m\#\CC\PP^{2n})$. Observe that
    \begin{equation}\label{Eq:RelationBasisElements}
      (\eta_j+\bar\eta_j)^{k} = 2\eta_j(\eta_j + \bar\eta_j)^{k-1}
      - (\eta_j-\bar\eta_j)(\eta_j+\bar\eta_j)^{k-1},
    \end{equation}
  thus
    \begin{equation}\label{Eq:TopBasisElement}
      \eta_j^{2n} =
      (\eta_j+\bar\eta_j)^{n} = 2\eta_j(\eta_j + \bar\eta_j)^{n-1}
      - (\eta_j-\bar\eta_j)(\eta_j+\bar\eta_j)^{n-1}.
    \end{equation}
  We set $\omega:=\eta_j^{2n}$ for any $j =1,\ldots,m$ and
  \begin{align*}
    e_j^k &:= \eta_j(\eta_j + \bar\eta_j)^{k},\quad j=1,\ldots,m,\quad k=0,\ldots,n-1\\
    f_j^k &:= (\eta_j-\bar\eta_j)(\eta_j + \bar\eta_j)^{k},\quad j=1,\ldots,m,\quad k=0,\ldots,n-1
  \end{align*}
  and in virtue of relation \eqref{Eq:TopBasisElement} the set
  \[
    B:=\{\omega\} \cup \{e_j^k \colon j=1,\ldots,m,\, k=0,\ldots,n-1\}
    \cup \{f_j^k \colon j=1,\ldots,m,\, k=0,\ldots,n-2\}
  \]
  is an integral basis of $\widetilde K^{}\left(m\#\CC\PP^{2n}\right)$.

  We proceed with the computation of $KO(m\#\CC\PP^{2n})$.  We have
  a long exact sequence for $\widetilde{KO}$-theory like in \eqref{Eq:LES}.
  From Theorem $2$ in \cite{Fujii1967} we deduce $\widetilde{KO}^{-1}\left(\CC\PP^{2n}\right)=0$ and therefore $\widetilde{KO}^{-1}\left(\vee_{j=1}^m \CC\PP^{2n}\right)=0$. Moreover $\widetilde{KO}^{1}\left(S^{4n}\right)=\widetilde{KO}^{-7}\left(S^{4n}\right)= \widetilde{KO}^{}\left(S^{4n+7}\right)=0$ by Bott periodicity. Hence we obtain a short exact sequence
   \begin{equation}\label{EQ:SESKOgroups}
      0 \lr \widetilde{KO}^{}\left(S^{4n}\right)
    \lr \widetilde{KO}^{}\left(m\#\CC\PP^{2n}\right)\lr
    \widetilde{KO}^{}\left(\vee_{j=1}^m \CC\PP^{2n-1}\right)
    \lr 0.
   \end{equation}
  Now we have to distinguish
  between the cases where $n$ is even or odd. We first assume that $n=2l$.
  In that case the ring $KO\left(\CC\PP^{2n-1}\right)$
  is isomorphic to $\ZZ[\eta_R]/\langle \eta_R^{n} \rangle$, see Theorem
  \ref{T:KtheoryOfComplexProjectiveSpace} (b). Hence all groups in \eqref{EQ:SESKOgroups} are
  torsion free. Therefore the kernel of $r\colon \widetilde K^{}\left(m\#\CC\PP^{2n}\right) \to \widetilde{KO}^{}\left(m\#\CC\PP^{2n}\right)$ is the same as the kernel of
  \[
    \varphi := c\circ r = 1+t
    \colon \widetilde K^{}\left(m\#\CC\PP^{2n}\right) \to \widetilde K^{}\left(m\#
    \CC\PP^{2n}\right)
  \]
  since $r\circ c = 2$ and thus $c$ is a monomorphism of the torsion free
  part of $\widetilde{KO}^{}\left(m\#\CC\PP^{2n}\right)$.

  Next we compute
  a basis of $\ker\varphi$. Using relation \eqref{Eq:RelationBasisElements}
  we have $\varphi(\omega) =2\omega$, $\varphi(e_j^k)
  =2e_j^k - f_j^k$ and $\varphi(f_j^k)=0$, thus if
  \[
    y = \lambda \omega + \sum_{j=1}^m \sum_{k=0}^{n-1}\lambda_j^k e^k_j
  \]
  then
  \begin{align*}
  \varphi(y) &= 2\lambda \omega + \sum_{j=1}^m\sum_{k=0}^{n-1} \lambda_j^k (2e_j^k - f_j^k) = \left(2\lambda + \sum_{j=1}^m \lambda_j^{n-1}\right)\omega + \sum_{j=1}^m\sum_{k=0}^{n-2} \lambda_j^k(2e_j^k - f_j^k)
  \end{align*}
where we used that $f_j^{n-1} = 2e_j^{n-1} - \omega$ by Equation \eqref{Eq:TopBasisElement}. As $\omega$ and $2e_j^k - f_j^k$, $j=1,\ldots,m$, $k=0,\ldots,n-2$, are linearly independent, we conclude that $\varphi(y)=0$ if and only if $\lambda^k_j=0$ for $j=1,\ldots,m$, $k=1,\ldots,n-2$ and
  \[
    \sum_{j=1}^m \lambda^{n-1}_j +2\lambda =0.
  \]
  This implies that the set
  \[
    \{f_j^k : j=1,\ldots,m,\quad k=0,\ldots,n-2\} \cup
    \{e^{n-1}_1-e^{n-1}_{j} : j= 2,\ldots,m\} \cup \{2e_1^{n-1} -\omega\},
  \]
  is an integral basis of $\ker\varphi$. Note that from \eqref{Eq:TopBasisElement} we have
   $2e_1^{n-1}-\omega = (\eta_1-\bar\eta_1)(\eta_1 + \bar\eta_1)^{n-1}$. By an inductive argument
   we see that
     \begin{equation}\label{Eq:KernelExpressedInwjk}
       (\eta_j -\bar\eta_j)(\eta_j+\bar\eta_j)^k = w_j^{k+1} + \text{linear combinations of }
       w_j^1,\ldots,w^k_j
     \end{equation}
   and
   \[
     e_1^{n-1} - e_j^{n-1}= \eta_1^{2n-1} -\eta_j^{2n-1}.
   \]
   Thus an integral basis of the kernel, in case $n$ is even, is given by
   \[
     \{w_j^k \colon j=1,\ldots,m,\, k=1,\ldots,n-1\} \cup
     \{w_1^{n}\}\cup
     \{\eta_1^{2n-1}-\eta_j^{2n-1}\colon j=2,\ldots,m\}.
    \]
    Now let us assume that $n= 2l+1$. Consider the commutative diagram
    \begin{center}
    \begin{tikzcd}
     0 \arrow{r}\arrow{d} & \widetilde K(S^{4n}) \arrow{r}{\pi^*} \arrow{d}{r_S}&
      \widetilde K^{}\left(m\#\CC\PP^{2n}\right)
      \arrow{r}{i^*}\arrow{d}{r_{\#}}& \widetilde K(\vee_{j=1}^m \CC\PP^{2n-1})
      \arrow{r}\arrow{d}{r_{\vee}}&  \arrow{d} 0 \\
    0 \arrow{r} & \widetilde{KO}(S^{4n}) \arrow{r}{\pi^*} & \widetilde KO^{}\left(m\#\CC\PP^{2n}\right)
      \arrow{r}{i^*}& \widetilde{KO}(\vee_{j=1}^m \CC\PP^{2n-1}) \arrow{r}&   0
    \end{tikzcd}
    \end{center}
    The map $r_S \colon \widetilde K^{}\left(S^{8l+4}\right) \to \widetilde{KO}^{}\left(S^{8l+4}\right)$
    is an isomorphism and therefore $i^*|_{\ker r_\#}\colon \ker r_\# \to \ker r_\vee$ is an
    isomorphism, hence the rank of $\ker r_\#$ is $mn$. We see that the set
    \[
      \{f_j^k : j=1,\ldots,m,\, k=0,\ldots,n-2\} \cup \{ 2e_j^{n-1} : j=1,\ldots,m \}
      \cup \{\omega\}
    \]
    is an integral basis of $(i^*)^{-1}\left( \ker r_\vee  \right)$, which follows because
    $e_j^{n-1}
    = \eta_j^{2n-1} -(n-1)\omega$ and the structure of the kernel of $r_\vee$, see Theorem
    \ref{T:KtheoryOfComplexProjectiveSpace} (d) (ii).
    The elements $f_j^k$ for $j=1,\ldots,m$ and
    $k=0,\ldots,n-2$ lie in the kernel of $r_\#$. Let
    \[
      y = \lambda \omega + \sum_{j=1}^m \lambda_j^{n-1}2 e_j^{n-1}
    \]
    for $\lambda,\lambda_j^{n-1} \in\ZZ$ and suppose $r_\#(y)=0$. From $\varphi(\omega) = 2\omega$ and $\varphi(e_j^{n-1}) = (\eta_j + \bar\eta_j)^n = \eta_j^{2n} = \omega$ it follows that
    \[
      \lambda + \sum_{j=1}^m \lambda_j^{n-1} =0.
    \]
    Hence $\ker r_\#$ is freely generated by the elements $f_j^k$ and $2e_j^{n-1} - \omega$. Observe
    from \eqref{Eq:TopBasisElement} that $2e_j^{n-1} - \omega =(\eta-\bar\eta)(\eta+\bar\eta)^{n-1}$. Thus in case of $n$ odd we deduce like in \eqref{Eq:KernelExpressedInwjk} that the kernel
    of $r_\#$ is freely generated by $w_j^k$ for $j=1,\ldots,m$ and $k=1,\ldots,n$.
\end{proof}
Hence by Equation \eqref{Eq:Kernel_of_r}, stable almost complex structures of $m\#\CC\PP^{2n}$ for
$n$ even are given by elements of the form
\begin{equation}\label{Eq:SACS}
  y = (2n+1)\sum_{i=1}^m \bar\eta_j+\sum_{j=1}^m\sum_{k=1}^{n-1} a^k_j w_j^k + a_1^n w_1^n +
  \sum_{j=2}^{m} b_j (\eta_1^{2n-1} - \eta_j^{2n-1}).
\end{equation}
and for $n$ odd
\begin{equation}\label{Eq:SACSnOdd}
  y = (2n+1)\sum_{i=1}^m \bar\eta_j+\sum_{j=1}^m\sum_{k=1}^{n} a^k_j w_j^k
\end{equation}
for $a_j^k, b_j \in \ZZ$. For Theorem
\ref{T:SACSandACS} we have to compute the $2n$--th Chern class $c_{2n}(E)$ of a vector bundle $E$
representing an element of the
form \eqref{Eq:SACS} and \eqref{Eq:SACSnOdd}. Let $\eta_1^{2n-1}-\eta_j^{2n-1}$ denote also a vector bundle over $m\#\CC\PP^{2n}$
which represents the element $\eta_1^{2n-1}-\eta_j^{2n-1}$ in $\widetilde K^{}\left(m\#\CC\PP^{2n}\right)$.
The total Chern class of $\eta_1^{2n-1} - \eta_j^{2n-1}$ can be computed through
the Chern character: We have
\[
  \widetilde{ch}(\eta_1^{2n-1} - \eta_j^{2n-1}) =\widetilde{ch}(\eta_1)^{2n-1}-
  \widetilde{ch}(\eta_j)^{2n-1} = x_1^{2n-1} - x_j^{2n-1}.
\]
The elements of degree $k$ in the Chern character are given by $\nu_k(c_1,\ldots,c_k)/k!$
where $\nu_k$ are
the Newton polynomials. The coefficient in front of $c_k$ in $\nu_k(c_1,\ldots,c_k)$ is $k$ (see
\cite{MR1122592}, p. 195) and the other terms are products of Chern classes of lower degree, hence
the only non vanishing Chern class is given by
\[
  c_{2n-1}(\eta_1^{2n-1}-\eta_j^{2n-1}) = (2n-2)!\; (x_1^{2n-1}-x_j^{2n-1})
\]
Thus the total Chern class of a vector bundle $E$ representing an element of the form \eqref{Eq:SACS}
is given by
\begin{align*}
  c(E) &= (1-(x_1+\ldots + x_m))^{2n+1}\\
  &\qquad \cdot \left( \frac{1+nx_1}{1-nx_1}  \right)^{a_1^n}
  \prod_{j=2}^m(1+ (2n-2)! (x_1^{2n-1} - x_j^{2n-1}))^{b_j}
    \prod_{j=1}^m\prod_{k=1}^{n-1} \left(\frac{1+kx_j}{1-kx_j}\right)^{a_j^k}
\end{align*}
and for \eqref{Eq:SACSnOdd}
\[
  c(E) = (1-(x_1+\ldots + x_m))^{2n+1}
    \prod_{j=1}^m\prod_{k=1}^{n} \left(\frac{1+kx_j}{1-kx_j}\right)^{a_j^k},
\]
where the coefficient in front of $x_1^{2n}=\ldots=x_m^{2n}$ is equal to $c_{2n}(E)$.
\begin{rem}
Note that for $m=1$ (and complex projective spaces of arbitrary dimension) this total Chern class was already computed by Thomas, see \cite[p.\ 130]{Thomas1974}.
\end{rem}
\section{Almost complex structures on $m\#
\CC\PP^{2n}$} \label{S:modd}

We now describe an explicit stable almost complex structure on $m\# \CC \PP^{2n}$, where $m=2u+1$, for which the assumptions of Theorem \ref{T:SACSandACS} are satisfied, thereby producing an almost complex structure on $m\# \CC \PP^{2n}$. We choose, in the notation of \eqref{Eq:SACS} and \eqref{Eq:SACSnOdd}, $a_{j}^k = 2$ for  $j=1,\ldots,u$ and $k=1$, and all other coefficients $0$. Then the top Chern class is as desired:

\begin{prop}\label{prop:modd} Let $m=2u+1$ be an odd number. In the cohomology ring of $m\#\CC\PP^{2n}$, the coefficient $c_{2n}$ of $x_1^{2n} = \cdots = x_m^{2n}$ of the class
\[
c = (1-(x_1+\cdots + x_{2u+1}))^{2n+1} \prod_{r=1}^u \left(\frac{1+x_r}{1-x_r}\right)^2
\]
is
\[
c_{2n} = m(2n-1) + 2 = \chi(m\# \CC\PP^{2n}).
\]
\end{prop}
\begin{proof}
As $x_i\cdot x_j=0$ for $i\neq j$, we have
\begin{align*}
(1-(x_1+\cdots + x_{2u+1}))^{2n+1} &= \sum_{j_0=0}^{2n+1} (-1)^{j_0}{2n+1\choose j_0} (x_1^{j_0} + \cdots + x_{2u+1}^{j_0})\\
&=\sum_{r=1}^{2u+1} \sum_{j_0=0}^{2n+1} (-1)^{j_0} {2n+1 \choose j_0} x_r^{j_0}.
\end{align*}
Thus,
\[
c = \prod_{r=1}^u (1-x_r)^{2n-1}(1+x_r)^2 \prod_{s=u+1}^{2u+1} (1-x_s)^{2n+1}.
\]
The factors $(1-x_s)^{2n+1}$ contribute $2n+1$ to $c_{2n}$, whereas the factors $(1-x_r)^{2n-1}(1+x_r)^2$ contribute $2n-3$. Thus,
\[
c_{2n} = u(2n-3) + (u+1)(2n+1) = (2u+1)(2n-1) + 2 = \chi((2u+1)\# \CC\PP^{2n}).
\]
\end{proof}

\bibliography{acsConnectedSumsBib}
\bibliographystyle{acm}

\end{document}